\documentclass[amstex,12pt, leqno]{amsart}
\usepackage{amsthm,amsmath,amsfonts,amssymb,enumerate,graphicx,hyperref,multicol,color,textcomp,hyperref,graphics}
\usepackage{bm,ifsym,float}
\newtheorem{lemma}{Lemma}
\newtheorem{theorem}{Theorem}
\newtheorem{remark}{Remark}

\numberwithin{equation}{section}

{\qed\bigskip}

{\thanks
{
\begin{footnotesize}
\hspace*{-7mm}
\end{footnotesize}
}}

\begin{document}

\leftline{ \scriptsize \it  }
\title [B\'{e}zier Variant of generalized Bernstein-Durrmeyer type operators]
{B\'{e}zier Variant of generalized Bernstein-Durrmeyer type operators}

\maketitle
\begin{center}
   Karunesh Kumar Singh\\
  Department of Applied Sciences and Humanities\\ Institute of Engineering and Technology\\
Lucknow-226021(Uttar Pradesh), India\\
Email: kksiitr.singh@gmail.com\\and\\

Asha Ram
Gairola\\
 Department of Mathematics,  Doon University\\
  Dehradun-248001 (Uttarakhand), India\\
  Email: ashagairola@gmail.com\\

\end{center}
\begin{abstract}
In this paper, we define   B\'{e}zier variant of generalized Bernstein-Durrmeyer type operators of second order, introduced by Ana et al. Then, we find an error estimate in terms of terms of Ditzian Totik modulus of smoothness. Next, we study rate of approximation for larger class function of bounded variation.
\end{abstract}
\section{Introduction}
 For $C[0,1],$ the space of continuous functions defined on $[0,1],$ and let $n\in \mathbb{N},$  the Bernstein polynomial operators $B_n(f):C[0,1]\longrightarrow C[0,1]$ are given by
\[B_n(f;x):=\sum_{k=0}^{n}p_{n,k}(x)f\left(\frac{k}{n}\right),\quad x\in[0,1],\]
where the basis functions for the above polynomials  $p_{n,k}(x)={n \choose k}x^k(1-x)^{n-k}$ for $0\le k\le n.$The fundamental polynomials satisfy following recurrence relation:

\begin{equation}\label{fundpoly}p_{n,k}(x)=(1-x)p_{n-1,k}(x)+xp_{n-1,k-1}(x),0<k<n.\end{equation}
  The Bernstein polynomial operators have some interesting properties eg. preservation of convexity, Lipschitz constant and  monotonicity etc. Although $B_n(f;x)$ are not suitable to approximate integrable function as such.
In order to approximate  bounded and integrable functions on $[0,1],$ Durrmeyer \cite{jld1967} and Lupa\c{s} \cite{AL1972} introduced its integral modifications as follows:
\begin{equation}\label{definitionoperator}D_n(f;x):=(n+1)\int_{0}^{1}\sum_{k=0}^{n}p_{n,k}(x)p_{n,k}(u)f(u)\,\textrm{d}u.
\end{equation}
These operators are obtained by a  suitable modification of Bernstein polynomial operators, wherein  the values $f(k/n)$ are replaced by the average values \[(n+1)\int_0^1 p_{n,k}(u)f(u)\,\textrm{d}u\] of $f$ in the interval $[0,1].$
The operators $D_n(f;x)$ have been extensively studied by
Derriennic \cite{MMD} and other authors (\cite{ZDKI1989},\cite{jld1967},\cite{HHGXLZ1991}),\cite{vg2009}, \cite{srivastavhmvgupta2003}. It is seen that the convergence of operators $D_n(f;x)$ to $f(x)$ is uniform. Also, it turns
out that the order of approximation by the operators $D_n(f;x)$ is at best, $O(n^{-1})$ however smooth the function may be.

\par
It is known that (\cite{vgrpa}) linear positive operators e.g. Bernstein polynomials operators, Baskakov operators, integral variants of Bernstein operators and others are saturated by order $n^{-1}.$ Consequently, it is not possible to improve the rate of convergence by these operators except for linear functions. Recently a new approach  to improve the rate of approximation by decomposition of the its weight function $p_{n,k}(x)$  was recently introduced by Khosravian-Arab et al.\cite{Khosravianetal}. The authors of \cite{Khosravianetal} applied this technique to the Bernstein polynomials operators and   were improved the order of approximation from linear order $O(n^{-1})$ to the  quadratic order $O(n^{-2})$ and cubic order $O(n^{-3})$ also. Subsequently, this technique was  also applied to the Bernstein-Durrmeyer operators $D_n(f;x)$ by Ana Maria et al. in \cite{ana2019}.

For $f\in L_B[0,1],$ Acu et al. \cite{ana2019} introduced generalized Bernstein Durrmeyer polynomial operators of order II as follows:
\begin{equation}\label{defoperator1}
D^{M,2}_n(f;x):=(n+1)\sum_{k=0}^{n}p_{n,k}^{M,2}(x)\int_{0}^{1}p_{n,k}(u)f(u)\,\textrm{d}u,n\ge 3
\end{equation}
where
\begin{equation}\label{fundpoly2}
p_{n,k}^{M,2}(x)=g(x,n)p_{n-2,k}(x)+h(x,n)p_{n-2,k-1}(x)+g(1-x,n)p_{n-2,k-2}(x),
\end{equation}and $p_{n-2,k}(x)=0$ if $k<0$ and $k>n-2$ and
the function
$g(x,n)=g_2(n)x^2+g_1(n)x+g_0(n)$ and $h(x,n)=h_0(n)x(1-x),$
 where  $g_i(n),$ $h_i(n)$ are the unknown sequences to be determined suitably by imposing the identities $D^{M,2}_n(e_i;x)=e_i(x)$ for $i=0,1$ and for $i=2$. So we get the condition $2g_2(n)-h_0(n)=0,$ $g_2(n)+2b_0(n)+b_1(n)=1.$ Particularly, if we take $g_2(n)=g_0(n)=1,$ $g_1(n)=-2$ and $h_0(n)=2.$ Then the fundamental polynomials (\ref{fundpoly2}) for the operators $D^{M,2}_n(f;x)$ is reduced to fundamental polynomials for Bernstein polynomials in (\ref{fundpoly}).
 Moreover, if we assume sequences such that  $g_0(n)=\frac{3}{2},g_1(n)=-n$ and $g_2(n)=n-2$ and $h_0(n)=2(n-2)$. Then we have,
 \begin{lemma}\cite{ana2019}\label{momentlemma}
We have
$$D^{M,2}_{n}((u-x);x)=0,$$
\begin{align*}D^{M,2}_{n}((u-x)^2;x)&=\frac{20x(1-x)}{(n+2)(n+3)}-\frac{3}{(n+2)(n+3)}.
\end{align*}
\end{lemma}
\begin{remark}
By an application of  lemma \ref{momentlemma}, we have
\begin{equation}\label{bdd2order}D^{M,2}_{n}((u-x)^2;x)\le \frac{C}{n+2}\varphi^2(x), x\in [0,1],
\end{equation} where $\varphi^2(x)=x(1-x)$ and $C$ is positive constant.
\end{remark}
Very recently,  Kajla and Acar \cite{KajlaandAcar} has applied this method to $\alpha-$ Bernstein operators. Motivated to study of various variants of these operators, we define the B\'{e}zier Variant of the operator $D^{M,2}_n(f;x)$ for a larger class, namely the class of functions of bounded variation.
For $\mu\ge 1$ and  $f\in L_B[0,1],$  we define B\'{e}zier variant of the operator $D^{M,2}_n(f;x)$ as
\begin{equation}\label{defoperator1}
D^{M,2}_{n,\mu}(f;x):=(n+1)\sum_{k=0}^{n}Q^{(\mu)}_{n,k}(x)\int_{0}^{1}p_{n,k}(u)f(u)\,\textrm{d}u,
\end{equation}
where
\begin{equation*}
Q^{(\mu)}_{n,k}(x)=(J_{n,k}(x))^{\mu}-(J_{n,k+1}(x))^{\mu}, J_{n,k}(x)=\sum_{j=k}^{n}K_{n,j}(x).
\end{equation*}
For $\mu=1,$ this family of B\'{e}zier operators yield the operators $D^{M,2}_{n,\mu}(f;x)$ studied by Acu et al. \cite{ana2019}.
Let $$W_{n,\mu}(x,u)=(n+1)\sum_{k=0}^nQ^{(\mu)}_{n,k}(x)p_{n,k}(u).$$
Then $D^{M,2}_{n,\mu}(f;x):=\int_{0}^{1}W_{n,\mu}(x,u)f(u)\textrm{d}u.$

\begin{lemma}
If $f\in C[0,1]$ then,  $\|D^{M,2}_{n}(f)\|\le \|f\|,$ where $\|.\|$ denotes the sup-norm on $[0,1].$
\end{lemma}
\begin{lemma}\label{lemmabddbezier}
If $f\in C[0,1]$ then,  $\|D^{M,2}_{n,\mu}(f)\|\le \mu \|f\|.$
\end{lemma}
\begin{proof}
Using the inequality $|a^\mu-b^\mu|\le \mu |a-b|$ with $0\le a,b \le 1, \mu \ge 1$ and definition of $D^{M,2}_{n,\mu}(f),$ we get for $\mu\ge 1$
\begin{equation}\label{lemma3result}
0<[(J_{n,k}(x))^{\mu}-(J_{n,k+1}(x))^{\mu}]\le \mu [J_{n,k}(x)-J_{n,k+1}(x)]=\mu K_{n,k}(x)
\end{equation} Applying $D^{M,2}_{n,\mu}(f;x)$ and Lemma \ref{lemmabddbezier}, we get
$$\|D^{M,2}_{n,\mu}(f)\|\le \mu \|D^{M,2}_{n}(f)\|\le \mu \|f\|.$$
\end{proof}

\section{Direct result}
Now, we derive a direct result for B\'{e}zier operators in terms of Ditzian Totik modulus of smoothness $\omega_{\varphi}(f,t)$.
\begin{theorem}
 Let  $\varphi(x)=\sqrt{x(1-x)},$ $\mu\ge 1$ and  $x\in[0,1].$ If $f\in C[0,1]$ then
$$|D^{M,2}_{n,\mu}(f;x)-f(x)| \le C \omega_{\varphi}\left(f,\sqrt{\frac{1}{n+2}}\right),$$
where $C$ is any absolute constant.\end{theorem}
\begin{proof}
We know that $g(t)=g(x)+\int\limits_x^tg'(u)\textrm{d}u.$ So
\begin{equation}\label{eqn11}|D^{M,2}_{n,\mu}(g;x)-g(x)|=\left|D^{M,2}_{n,\mu}\left(\int\limits_x^tg'(u)\textrm{d}u,x\right)\right|.\end{equation}
For any  $x,t\in(0,1),$ we have
\begin{equation}\label{eqn12}\left|\int\limits_x^tg'(u)\textrm{d}u\right|\le \|\varphi g'\|\left|\int\limits_x^t\frac{1}{\varphi(u)}\textrm{d}u\right|.\end{equation}
Therefore,
\begin{eqnarray} \label{eqn13}\left|\int\limits_x^t\frac{1}{\varphi(u)}\textrm{d}u\right|&=&\left|\int\limits_x^t\frac{1}{\sqrt{u(1-u)}}\textrm{d}u\right|\nonumber\\
&\le &\left|\int\limits_x^t\left(\frac{1}{\sqrt{u}}+\frac{1}{\sqrt{1-u}}\right)\textrm{d}u\right|\nonumber\\
&\le & 2\left(|\sqrt{t}-\sqrt{x}|+|\sqrt{1-t}-\sqrt{1-x}|\right)\nonumber\\
&=&2|t-x|\left(\frac{1}{\sqrt{t}+\sqrt{x}}+\frac{1}{\sqrt{1-t}+\sqrt{1-x}}\right)\nonumber\\
&<&2|t-x|\left(\frac{1}{\sqrt{x}}+\frac{1}{\sqrt{1-x}}\right)\nonumber\\
&<&\frac{2\sqrt{2}|t-x|}{\varphi(x)}.
\end{eqnarray}
Combining (\ref{eqn11})-(\ref{eqn13}) and using Cauchy-Schwarz inequality, we have that
\begin{eqnarray*}\label{eqn14}|D^{M,2}_{n,\mu}(g;x)-g(x)| &\le& 2\sqrt{2}\|\varphi g'\|\varphi^{-1}(x)D^{M,2}_{n,\mu}(|t-x|;x)\\
&\le& 2 \sqrt{2}\|\varphi g'\|\varphi^{-1}(x)\left(D^{M,2}_{n,\mu}((t-x)^2;x)\right)^{1/2}.
\end{eqnarray*}
Now using inequality (\ref{bdd2order}), we obtain
\begin{equation}\label{eqn14}
|D^{M,2}_{n,\mu}(g;x)-g(x)|<C\sqrt{\frac{1}{(n+2)}}\|\varphi g'\|.\end{equation}
Using Lemma \ref{momentlemma} and inequality (\ref{eqn14}), we obtain
\begin{eqnarray}\label{eqn15}|D^{M,2}_{n,\mu}(f;x)-f| &\le& |D^{M,2}_{n,\mu}(f-g;x)|+|f-g|+|D^{M,1}_{n,\mu}(g;x)-g(x)|\nonumber\\
&\le  & C \left(\|f-g\|+\sqrt{\frac{1}{(n+2)}}\|\varphi g'\|\right).
\end{eqnarray}
Taking infimum on both sides of (\ref{eqn15}) over all $g \in W^2_{\phi}$, we reach to
\begin{equation*}\label{eqn16}|D^{M,2}_{n,\mu}(f;x)-f|\le C K_{\varphi}\left(f,\sqrt{\frac{1}{(n+2)}}\right)
\end{equation*}
Using relation  $K_{\varphi}(f,t) \sim \omega_{\varphi}(f,t),$  we get the required result.
\end{proof}
Ozarslan and Aktuglu \cite{ozarslan} is considered the Lipschitz-type space with two parameters $\alpha_1\ge 0,\alpha_2>0$, which is defined as
$$Lip_M^{(\alpha_1,\alpha_2)}(\zeta)=\left\{f\in C[0,1]:|f(t)-f(x)|\le C_0\frac{|t-x|^\zeta}{(t+\alpha_1x^2+\alpha_2x)^{\zeta/2}}:t\in[0,1], x\in (0,1)\right\},$$
where $0<\zeta\le 1,$ and $C_0$ is any absolute constant.
\begin{theorem}
Let $f\in Lip_M^{(\alpha_1,\alpha_2)}(\zeta).$ Then for all $x\in (0,1]$, there holds:
\begin{equation*}\label{eqn17}|D^{M,2}_{n,\mu}(f;x)-f|\le C \left(\frac{\mu \varphi^2(x)}{(n+2)(\alpha_1x^2+\alpha_2x)}\right)^{\frac{\zeta}{2}},
\end{equation*} where $C$ is any absolute constant.
\end{theorem}
\begin{proof}
Using Lemma \ref{momentlemma} and (\ref{lemma3result}) and H\"{o}lder's inequality with $p=\frac{2}{\zeta}$ and $p=\frac{2}{2-\zeta},$ we get
\begin{align*}
 &|D^{M,2}_{n,\mu}(f;x)-f(x)|\\
 &\le (n+1)\sum_{k=0}^{n}Q^{(\mu)}_{n,k}(x)\int_{0}^{1}p_{n,k}(u)|f(u)-f(x)|\,\textrm{d}u\\
 &\le (n+1)\sum_{k=0}^{n}Q^{(\mu)}_{n,k}(x)\left(\int_{0}^{1}p_{n,k}(u)|f(u)-f(x)|^{\frac{2}{\zeta}}\,\textrm{d}u\right)^{\frac{\zeta}{2}}\\
 &\le  \left[(n+1)\sum_{k=0}^{n}Q^{(\mu)}_{n,k}(x)\int_{0}^{1}p_{n,k}(u)|f(u)-f(x)|^{\frac{2}{\zeta}}\,\textrm{d}u\right]^{\frac{\zeta}{2}}\times\\
&\left[(n+1)\sum_{k=0}^{n}Q^{(\mu)}_{n,k}(x)\int_{0}^{1}p_{n,k}(u)\,\textrm{d}u\right]^{\frac{2-\zeta}{2}}\\
&=\left[(n+1)\sum_{k=0}^{n}Q^{(\mu)}_{n,k}(x)\int_{0}^{1}p_{n,k}(u)|f(u)-f(x)|^{\frac{2}{\zeta}}\,\textrm{d}u\right]^{\frac{\zeta}{2}}\\
&\le C\left((n+1)\sum_{k=0}^{n}Q^{(\mu)}_{n,k}(x)\int_{0}^{1}p_{n,k}(u)\frac{(u-x)^2}{(u+\alpha_1x^2+\alpha_2x)}\textrm{d}u\right)^{\frac{\zeta}{2}}\\
&\le \frac{C}{(\alpha_1x^2+\alpha_2x)^{\frac{\zeta}{2}}}  \left((n+1)\sum_{k=0}^{n}Q^{(\mu)}_{n,k}(x)\int_{0}^{1}p_{n,k}(u)(u-x)^2\textrm{d}u\right)^{\frac{\zeta}{2}}\\
&\le \frac{C}{(\alpha_1x^2+\alpha_2x)^{\frac{\zeta}{2}}} [D^{M,2}_{n,\mu}((u-x)^2;x)]^{\frac{\zeta}{2}}\\
&\le C \left(\frac{\mu \varphi^(x)}{(n+2)(\alpha_1x^2+\alpha_2x)}\right)^{\frac{\zeta}{2}}.
\end{align*}
Hence, the desired result follows.
\end{proof}
\section{Rate of convergence}
If we define $$\kappa_{n,\mu}(x,y)=\int\limits_0^yW_{n,k}(x,t)\textrm{d}t.$$ Then, it is obvious that $\kappa_{n,\mu}(x,1)=1$.
\begin{lemma}\label{lem3}
Let $x\in (0,1)$ and $C>2$ then for sufficiently large $n$ we have
$$\kappa_{n,\mu}(x,y)=\int\limits_0^yW_{n,\mu}(x,t)dt\le \frac{C\mu x(1-x)}{n(x-y)^2},0<y<x$$
$$1-\kappa_{n,\mu}(x,z)=\int\limits_z^1W_{n,\mu}(x,t)dt\le \frac{C\mu x(1-x)}{n(z-x)^2},x<z<1.$$
\end{lemma}
\begin{remark} We have
\begin{align*}\kappa_{n,\mu}(x,y)&=\int\limits_0^yW_{n,\mu}(x,t)dt \le \int\limits_0^yW_{n,\mu}(x,t)\frac{(t-x)^2}{(y-x)^2}dt\\&=\frac{D^{M,2}_{n,\mu}((t-x)^2;x)}{(y-x)^2}\le
 \frac{\mu D^{M,2}_n((t-x)^2;x)}{(y-x)^2}\le\frac{\mu }{n+2}.\frac{\varphi^2(x)}{(y-x)^2}.\end{align*}
\end{remark}
 Now, we discuss the approximation properties of functions having derivative of bounded variation on $[0,1]$. Let $\mathbb{B}[0,1]$ denote the class of differentiable functions $g$ defined on $[0,1]$, whose derivative $g'$ is of bounded variation on $[0,1]$. The functions $g\in \mathbb{B}[0,1]$ is expressed as $g(x)=\int_0^x h(t)dt+g(0)$, where $h\in \mathbb{B}[0,1],$ i.e., $h$ is a function of bounded variation on $[0,1].$

\begin{theorem}
Let $f\in\mathbb{B}[0,1]$ then for $\mu\ge 1, 0<x<1 $ a and sufficiently large $n$ we have
\begin{align*}\left|D^{M,2}_{n,\mu}(f;x)-f(x)\right|&\le \left(\frac{1}{\mu+1}|f'(x+)+\mu f'(x-)|+|f'(x+)-f'(x-)|\right)\frac{\mu}{n+2}\varphi^2(x)\\
&+\frac{\mu}{n+2}.\frac{\varphi^2(x)}{x^2}\sum_{k=1}^{\sqrt{n}}\left(\bigvee_{x-\frac{x}{k}}^x(f')_x\right)
+\frac{x}{\sqrt{n}}\left(\bigvee_{x}^{x+\frac{1-x}{\sqrt{n}}}(f')_x\right) \left(\bigvee_{x-\frac{x}{k}}^x(f')_x\right)\\
&+\frac{\mu}{n+2}.\frac{\varphi^2(x)}{1-x}\sum_{k=1}^{\sqrt{n}}\left(\bigvee_x^{x+\frac{1-x}{k}}(f')_x\right)
+\frac{1-x}{\sqrt{n}}\left(\bigvee_{x}^{x+\frac{1-x}{\sqrt{n}}}(f')_x\right).
    \end{align*}
\[(f')_x(t)=\left\{
\begin{array}{lll}
f'(t)-f'(x-),  & \hbox{$0\le t <x$} \\
0,& \hbox{$t=x$}\\
f'(t)-f'(x+),& \hbox{$x<t\le1$.}
\end{array}\right.\]
\end{theorem}
\begin{proof}
  As we know that $D^{M,2}_{n,\mu}(1;x)=1.$ Therefore, we have
\begin{eqnarray}\label{eqn0000}
D^{M,2}_{n,\mu}(f;x)-f(x)&=&\int\limits_0^1W_{n,\mu}(x,u)(f(u)-f(x))\,\textrm{d}u\nonumber\\
&=&\int\limits_0^1W_{n,\mu}(x,u)\int\limits_x^uf'(t)\,dt\,\textrm{d}u
\end{eqnarray}
Since $f\in \mathbb{B}[0,1]$, we may write
\begin{eqnarray}\label{eqn1111}
f'(t)&=&\frac{f'(x+)+\mu f'(x-)}{\mu+1}+(f')_x(t)+\frac{f'(x+)- f'(x-)}{2}\left(\mbox{sign}(t-x)+\frac{\mu-1}{\mu+1}\right)\nonumber\\&+&\delta_x(t)
\left(f'(t)-\frac{f'(x+)+\mu f'(x-)}{2}\right)
\end{eqnarray} where
\[\mbox{sign}(t)=\left\{
\begin{array}{lll}
1,  & \hbox{$t>0$} \\
0,& \hbox{$t=0$}\\
-1,& \hbox{$t<0$.}
\end{array}\right.\]
and
\[\delta_x(t)=\left\{
\begin{array}{ll}
1,  & \hbox{$t=x$} \\
0,& \hbox{$t\neq x$.}
\end{array}\right.\]
Putting  the value of $f'(t)$ form (\ref{eqn1111}) in \ref{eqn0000}, we get estimates corresponding to four terms of (\ref{eqn1111}), say $K_1, K_2, K_3$ and $K_4$ respectively. Obviously,
$$K_4=\int\limits_0^1\left(\int\limits_x^u  \left(f'(t)-\frac{f'(x+)+\mu f'(x-)}{2}\right) \delta_x(t)dt\right) W_{n,\mu}(x,u)\textrm{d}u=0.$$
Now,
Using Cauchy's Schwarz inequality and Lemma \ref{lem3}, we obtain
\begin{align*}
K_1=\int\limits_0^1\left(\int\limits_x^u \frac{f'(x+)+\mu f'(x-)}{\mu+1}dt\right)W_{n,\mu}(x,u)\textrm{d}u&=\frac{f'(x+)+\mu f'(x-)}{\mu+1}\int\limits_0^1(u-x)W_{n,\mu}(x,u)\textrm{d}u\\
&=\frac{f'(x+)+\mu f'(x-)}{\mu+1}D^{M,1}_{n,\mu}((u-x);x)\\
&\le \frac{1}{\sqrt{n+2}} \frac{f'(x+)+\mu f'(x-)}{\mu+1}\varphi(x).
\end{align*}
Next, again using Cauchy's Schwarz inequality and Lemma \ref{lem3}, we obtain
\begin{align*}K_3&=\int\limits_0^1\left(\int\limits_x^u \left(\frac{f'(x+)- f'(x-)}{2}\right)\left(\mbox{sign}(t-x)+\frac{\mu-1}{\mu+1}\right)dt\right)
W_{n,\mu}(x,u)\textrm{d}u\\
&=\left(\frac{f'(x+)- f'(x-)}{2}\right)\Bigg[-\int\limits_0^x\left(\int\limits_u^x \left(\mbox{sign}(t-x)+\frac{\mu-1}{\mu+1}\right)dt\right)W_{n,\mu}(x,u)\textrm{d}u\\
&+\int\limits_x^1\left(\int\limits_x^u \left(\mbox{sign}_{\mu}(t-x)+\frac{\mu-1}{\mu+1}\right)dt\right)
W_{n,\mu}(x,u)\textrm{d}u\Bigg]\\
&\le \left|f'(x+)- f'(x-)\right|\int\limits_0^1|u-x|W_{n,\mu}(x,u)\textrm{d}u\\
&\le \left|f'(x+)- f'(x-)\right|D^{M,1}_{n,\mu}(|u-x|;x)\\
&\le \frac{1}{\sqrt{n+2}} \left|f'(x+)- f'(x-)\right|\varphi(x).
\end{align*}
Now, we estimate $K_2$ as follows:
\begin{align*}
K_2=\int\limits_0^1\left(\int\limits_x^u (f')_x(t)dt\right)W_{n,\mu}(x,u)\textrm{d}u&=\int\limits_0^x\left(\int\limits_x^u (f')_x(t)dt\right) W_{n,\mu}(x,u)\textrm{d}u\\
&+\int\limits_x^1\left(\int\limits_x^u (f')_x(t)dt\right)
W_{n,\mu}(x,u)\textrm{d}u\\
&=K_5+K_6,\mbox{say}.
\end{align*}
Using Lemma \ref{lem3} and definition of $\kappa_{n,\mu}(x,u),$ we may write
\begin{align*}K_5&=\int\limits_0^x\left(\int\limits_x^u (f')_x(t)dt\right) \frac{d}{\textrm{d}u}\kappa_{n,\mu}(x,u)\textrm{d}u\\
\end{align*}
  Integrating by parts, we obtain
  \begin{align*}|K_5|&\le \int\limits_0^x|(f')_x(u)|\kappa_{n,\mu}(x,u)\textrm{d}u\\
&\le \int\limits_0^{x-x/\sqrt{n}}|(f')_x(u)|\kappa_{n,\mu}(x,u)\textrm{d}u+\int\limits_{x-x/\sqrt{n}}^x|(f')_x(u)|\kappa_{n,\mu}(x,u)\textrm{d}u\\
&=K_7+K_8,\mbox{say}.
\end{align*}
  In view of facts $(f')_x(x)=0$ and $\kappa_{n,\mu}(x,u)\le 1,$  we get
  \begin{align*}K_8&= \int\limits_0^x|(f')_x(u)-(f')_x(x)|\kappa_{n,\mu}(x,u)\textrm{d}u\\
&\le \int\limits_{x-x/\sqrt{n}}^x\left(\bigvee_{u}^x(f')_x\right)\textrm{d}u\\
&\le \left(\bigvee_{u}^x(f')_x\right) \int\limits_{x-x/\sqrt{n}}^x \textrm{d}u=\frac{x}{\sqrt{n}}\left(\bigvee_{u}^x(f')_x\right).
\end{align*}
Using Lemma \ref{lem3}, definition of $\kappa_{n,\mu}(x,u),$ and transformation $u=x-\frac{x}{t}$ we may write
   \begin{align*}K_7&\le \frac{\mu}{n+2}\varphi^2(x) \int\limits_0^{x-x/\sqrt{n}}|(f')_x(u)-(f')_x(x)|\frac{\textrm{d}u}{(u-x)^2}\\
  & \le \frac{\mu }{n+2}\varphi^2(x) \int\limits_0^{x-x/\sqrt{n}}\left(\bigvee_{u}^x(f')_x\right)\frac{\textrm{d}u}{(u-x)^2}\\
    &\le \frac{\mu}{n+2}\frac{\varphi^2(x) }{x^2} \int\limits_1^{\sqrt{n}} \left(\bigvee_{x-\frac{x}{t}}^x(f')_x\right)dt\\
   &\le \frac{\mu }{n+2}\frac{\varphi^2(x) }{x^2}\sum_{k=1}^{|\sqrt{n}|}\left(\bigvee_{x-\frac{x}{t}}^x(f')_x\right).
\end{align*}
 Combining the estimates of $I_7$ and $I_8,$ we have
    \begin{align*}|K_5|&\le \frac{\mu }{n+2}\frac{\varphi^2(x) }{x^2}\sum_{k=1}^{|\sqrt{n}|}\left(\bigvee_{x-\frac{x}{t}}^x(f')_x\right)
   + \frac{x}{\sqrt{n}}\left(\bigvee_{u}^x(f')_x\right)
   \end{align*}
   In order to estimate $K_6,$ we use integration by parts, Lemma \ref{lem3} and transformation $z=x+\frac{1-x}{\sqrt{n}}.$ Therefore,we proceed as follows:
   \begin{align*}|K_6|&=\left|\int\limits_x^1\left(\int\limits_x^u (f')_x(t)dt\right)W_{n,\mu}(x,u)\textrm{d}u\right|\\
  &=\Bigg|\int\limits_x^z\left(\int\limits_x^u (f')_x(t)dt\right)\frac{\partial}{\partial u}(1-\kappa_{n,\mu}(x,u))\textrm{d}u+\int\limits_z^1\left(\int\limits_x^u (f')_x(t)dt\right)\frac{\partial}{\partial u}(1-\kappa_{n,\mu}(x,u))\textrm{d}u\Bigg|\\
   &=\Bigg|\left[\int\limits_x^u (f')_x(t)dt(1-\kappa_{n,\mu}(x,u))\right]_{x}^{z}-\int\limits_x^z (f')_x(u)(1-\kappa_{n,\mu}(x,u))\textrm{d}u\\
   &+\left[\int\limits_x^u (f')_x(t)dt(1-\kappa_{n,\mu}(x,u)) \right]_{z}^{1}-\int\limits_z^1 (f')_x(u)(1-\kappa_{n,\mu}(x,u))\textrm{d}u\Bigg|\\
   &=\left|\int\limits_x^z (f')_x(u)(1-\kappa_{n,\mu}(x,u))\textrm{d}u+\int\limits_z^1 (f')_x(u)(1-\kappa_{n,\mu}(x,u))\textrm{d}u\right|\\
   &\le \frac{\mu}{n+2}.\varphi^2(x)  \int\limits_z^1\left(\bigvee_{x}^u(f')_x\right)(u-x)^{-2}\,\textrm{d}u+\int\limits_x^z
    \left(\bigvee_{x}^u(f')_x\right)\,\textrm{d}u\\
    &\le \frac{\mu}{n+2}.\varphi^2(x) \int\limits_{x+\frac{1-x}{\sqrt{n}}}^1\left(\bigvee_{x}^u(f')_x\right)(u-x)^{-2}\,\textrm{d}u
    +\frac{1-x}{\sqrt{n}}
    \left(\bigvee_{x}^{x+\frac{1-x}{\sqrt{n}}}(f')_x\right).
    \end{align*}
Now, substituting  $t=\frac{1-x}{t-x},$ we get
\begin{align*}|K_6| &\le \frac{\mu}{n+2}.\varphi^2(x) \int\limits_{1}^{\sqrt{n}}\left(\bigvee_{x}^{x+\frac{1-x}{t}}(f')_x\right)(1-x)^{-1}\,dt+
\frac{1-x}{\sqrt{n}}\left(\bigvee_{x}^{x+\frac{1-x}{\sqrt{n}}}(f')_x\right)\\
   &\le \frac{\mu }{n+2}\frac{\varphi^2(x)}{1-x} \sum_{k=1}^{\sqrt{n}}\left(\bigvee_{x}^{x+\frac{1-x}{k}}(f')_x\right)+\frac{1-x}{\sqrt{n}}\left(\bigvee_{x}^{x+\frac{1-x}{\sqrt{n}}}(f')_x\right).
   \end{align*}
Combining the estimates of $K_1-K_8,$ we get the desired result. Hence the proof follows.
\end{proof}

\end{document}